\documentclass{amsart}

\usepackage{amsmath}
\usepackage{amssymb}
\usepackage{amsthm}

\usepackage[dvipsnames]{xcolor}
\usepackage[colorlinks,linkcolor=MidnightBlue,citecolor=MidnightBlue]{hyperref}

\usepackage[margin=1.0in]{geometry}               
\usepackage{tikz}
\tikzstyle{ball} = [circle,shading=ball, ball color=black,
    minimum size=1mm,inner sep=1.3pt]
    \tikzstyle{dot} = [circle,fill=black,
    minimum size=3.5pt,inner sep=0pt]
\tikzstyle{miniball} = [circle,shading=ball, ball color=black,
    minimum size=1mm,inner sep=0.5pt]
\usepackage{tikz-cd}

\usepackage{enumitem}

\newtheorem{thm}{Theorem}[section]
\newtheorem{lemma}[thm]{Lemma}

\newtheorem{cor}[thm]{Corollary}
\newtheorem{prop}[thm]{Proposition}

\newtheorem{Definition}[thm]{Definition}
\newenvironment{defn}
  {\begin{Definition}\rm}{\end{Definition}}

\newtheorem{Example}[thm]{Example}
\newenvironment{example}
  {\begin{Example}\rm}{\end{Example}}

\newtheorem{Remark}[thm]{Remark}

\DeclareMathOperator{\av}{av}
\DeclareMathOperator{\bound}{Bound}
\DeclareMathOperator{\cb}{CB}
\DeclareMathOperator{\wt}{wt}
\DeclareMathOperator{\rk}{rk}
\DeclareMathOperator{\gr}{Gr}

\newcommand{\Z}{\mathbb{Z}}

\newcommand{\tS}{\tilde{S}_n}
\newcommand{\tSk}{\tilde{S}_n^k}
\newcommand{\tSz}{\tilde{S}_n^{0}}

\title{Counting weighted maximal chains in the circular Bruhat order}
\author{Gopal Goel}
\address{M.I.T., Cambridge, MA}
\email{gopalkg@mit.edu}
\author{Olivia McGough}
\address{Reed College, Portland, OR}
\email{mcgougho@reed.edu}
\thanks{The second author's work is partially supported by the Reed College Robert \& Louise
Rosenbaum Fellowship.}
\author{David Perkinson}
\address{Reed College, Portland, OR}
\email{davidp@reed.edu}



\begin{document}

\begin{abstract} 
  The totally nonnegative Grassmannian $\gr(k,n)_{\geq0}$ is the subset of the
  real Grassmannian~$\gr(k,n)$ consisting of points with all nonnegative Pl\"ucker
  coordinates.  The circular Bruhat order is a poset isomorphic to the face
  poset of Postnikov's (\cite{Postnikov}) positroid cell decomposition
  of~$\gr_{\geq0}(k,n)$.  We provide a closed formula for the sum of 
  its weighted chains in the spirit of Stembridge (\cite{Stembridge}).
\end{abstract}

\maketitle

\section{Introduction.}\label{sect:intro}
Let~$S_n$ be the symmetric group on $[n]:=\left\{1,\ldots,n\right\}$.  An
\emph{inversion} of~$\pi\in S_n$ is a pair~$i,j\in[n]$ such that~$i<j$ and
$\pi(i)>\pi(j)$.  The number of inversions of~$\pi$ is its \emph{length},
denoted $\ell(\pi)$.  The \emph{Bruhat order} on~$S_n$ is a partial ordering
on~$S_n$, graded by length. Its cover relations have the form~$\pi
s_{ij}\lessdot\pi$ where~$s_{ij}:=(i,j)$ is a transposition such that
$\ell(\pi)=\ell(\pi s_{ij})+1$.  The maximal element of the Bruhat order,
written in row notation, is $\pi_{\mathrm{top}}=[n,n-1,\ldots,1]$ of
length~$r:=\binom{n}{2}$ and the smallest element is the identity
permutation~$\mathrm{id}=[1,2,\ldots,n]$ of length~$0$.  In the Bruhat order,
each maximal chain has the
form~$\mathrm{id}=\pi_0\lessdot\pi_1\lessdot\ldots\lessdot\pi_{r}=\pi_{\mathrm{top}}$.
Let $\alpha_1,\ldots,\alpha_n$ be indeterminates. Define the \emph{weight} of a
covering~$\pi s_{ij}\lessdot \pi$ with~$i<j$ to
be~$\alpha_{i}+\alpha_{i+1}+\cdots+\alpha_{j-1}$, and then define the weight of
a maximal chain to be the product of the weights of its cover relations.  In a
result that extends to all Weyl groups, Stembridge (\cite{Stembridge}) shows
that the sum of the weights of the maximal chains is
\[
  \frac{\binom{n}{2}!}{1^{n-1}2^{n-2}\cdots(n-1)^{1}}\prod_{1\leq i<j\leq
  n}(\alpha_i+\cdots+\alpha_{j-1}).
\]
For instance, this formula reduces to~$\binom{n}{2}!$ after
setting all weights~$\alpha_i=1$.

The \emph{totally nonnegative Grassmannian} $\gr(k,n)_{\geq0}$ is introduced in
\cite{Postnikov} as the subset of points in the real Grassmannian~$\gr(k,n)$
which have all nonnegative Pl\"ucker coordinates.  It is related to areas as
diverse as cluster algebras (\cite{GL}),  electrical networks (\cite{Lam}),
solitons (\cite{KW}), scattering amplitudes in Yang-Mills theory
(\cite{Arkani}), and the mathematical theory of juggling (\cite{KLS}).
Postnikov gave a decomposition of~$\gr(k,n)_{\geq0}$ into \emph{positroid cells}
defined by setting certain Pl\"ucker coordinates equal to zero, and he
conjectured that this decomposition forms a regular CW-complex.  A
generalization of that conjecture due to Williams (\cite{Williams2}) was proved
by Galashin, Karp, and Lam (\cite{GKL}).  Our object of interest is the face
poset of this complex, known as the \emph{circular Bruhat order} (\cite[Section
17]{Postnikov}).  Postnikov's work provides characterizations in terms of many
different combinatorial objects, e.g.,  decorated permutations, Grassmannian
necklaces, Le-diagrams, and  equivalence classes of certain plabic (planar,
bi-colored) graphs.  The list is extended by Knutson, Lam, and Speyer
(\cite{KLS}) to include bounded affine permutations, bounded juggling patterns,
and equivalence classes of intervals in the~$k$-Bruhat order for~$S_n$.  In this
paper, we use the language of bounded affine permutations.

Our purpose is to give a Stembridge-like formula for the circular Bruhat order.
We define ``circular'' analogues of Stembridge's weights
(Definition~\ref{def:weights}) and our main result, Theorem~\ref{thm:main},
provides a closed formula for the sum of the weights of the maximal chains in the
circular Bruhat order:
\[
  f(k,n)(\alpha_1+\cdots+\alpha_{n})^{k(n-k)},
\]
where~$f(k,n)$ is the number of Young tableaux for the~$k\times (n-k)$ rectangle
(cf.~Example~\ref{example:cb(2,3)}).  Section~\ref{sect:background} provides
background and notation.  Section~\ref{sect:main} states and proves the main
result, Theorem~\ref{thm:main}.  The proof relies on two technical lemmas whose
proofs are relegated to Section~\ref{sect:lemmas}.  
These proofs rely on the
interpretation of the circular Bruhat order in terms of intervals in
the~$k$-Bruhat order for~$S_n$ developed in \cite{KLS}.  We also use a result of
Bergeron and Sottile (\cite[Corollary 1.3.1]{BS}) on cyclic shifts of~$k$-Bruhat
intervals.  Their proof is a consequence of a symmetry they find for
Littlewood-Richardson coefficients using geometry.  It would be nice to have a
purely combinatorial proof of their cyclic shift result.

\subsection*{Acknowledgments} We would like to thank Alex Postnikov for
suggesting this project.

\section{Circular Bruhat order}\label{sect:background}
We recall ideas and notation introduced in \cite{KLS}, which built on earlier
work by Postnikov on the totally nonnegative Grassmannian (\cite{Postnikov}).
Our reference for the affine symmetric group is~\cite{BB}.
Let~$\tS$ denote the group of \emph{affine permutations} consisting of
bijections~$f\colon\Z\to\Z$ satisfying~$f(i+n)=f(i)+n$ for all~$i\in\Z$.  We use
the standard \emph{window} notation~$f=[f(1),f(2),\ldots,f(n)]$ to
represent~$f\in\tS$.  Define the averaging function on~$\tS$
by~$\av(f)=\frac{1}{n}\sum_{i=1}^{n}(f(i)-i)$, and for~$0\leq k\leq n$,
let~$\tSk:=\av^{-1}(k)$.  In particular,~$\tSz$ is the \emph{affine
symmetric group}.

The affine symmetric group is generated by its \emph{simple reflections}:
\[
  s_i =
  \begin{cases}
    [0,2,3,\ldots,n-1,n+1]&\text{if $i=0$,}\\
    [1,2,\ldots,i-1,i+1,i,i+2,\ldots,n]&\text{if $0<i\leq n$.}
  \end{cases}
\]
For instance,
\[
  [f(1),\ldots,f(n)]s_0=[f(0),f(2),\ldots,f(n-1),f(n+1)]=[f(n)-n,f(2),\ldots,f(n-1),f(1)+n].
\]
Then~$(\tS^{0},\left\{s_0,\ldots,s_{n-1}\right\})$ is the affine Coxeter
group~$\tilde{A}_{n-1}$ and is thus a graded poset under the Bruhat order.  The
\emph{reflections} for~$\tilde{A}_{n-1}$, i.e., the conjugates of the simple
reflections, are 
\begin{equation}\label{eqn:reflections}
  [1,2,\ldots,i-1,j-rn,i+1,\ldots,j-1,i+rn,j+1,\ldots,n]
\end{equation}
for~$1\leq i<j\leq n$ and~$r\in\Z$.

The mapping~$[f(1),\ldots,f(n)]\mapsto[f(1)-k,\ldots,f(n)-k]$ is a
bijection~$\tSk\to\tSz$, and thus the Bruhat order on~$\tSz$ induces a graded
poset structure on~$\tSk$ for which we now give an explicit description.  A
pair~$(i,j)\in\Z^2$ is an \emph{inversion} for~$f\in\tSk$ if~$i<j$
and~$f(j)>f(i)$.  Define an equivalence relation on the set of inversions
by~$(i,j)\sim(i',j')$ if~$i'=i+r n$ and~$j'=j+r n$ for some integer~$r$.  Then
the \emph{length} of~$f$, denoted~$\ell(f)$, is the number of equivalence
classes of inversions of~$f$.  (This notion of length coincides with that
inherited from the Bruhat order (\cite[Proposition~4.1]{BB}).)  In general,
if~$f\in\tS^{k'}$ and $g\in\tS^{k}$, then~$fg\in\tS^{k'+k}$.  In
particular,~$\tSz$ acts on~$\tSk$.  If~$f,g\in\tSk$, then~$f$ covers~$g$,
denoted $g\lessdot f$, exactly when $g=ft$ from some reflection~$t$
from~$\tilde{A}_{n-1}$ and $\ell(f)=\ell(g)+1$.

A permutation~$f\in\tS$ is \emph{bounded} if~$i\leq f(i)\leq i+n$ for
all~$i\in\Z$.  For each~$0\leq k\leq n$ the bounded elements of~$\tSk$ are
denoted 
\[
  \bound(k,n):=\left\{f\in\tSk:i\leq f(i)\leq i+n \text{ for all
  $i\in\Z$}\right\}.
\]
By Lemma~3.6 of~\cite{KLS},
$\bound(k,n)$ is a lower order ideal in~$\tSk$ and thus forms a graded
poset with rank function given by length.  The
\emph{dual} of a poset~$P=(P,<_P)$, is the poset~$P^*=(P,<_{P^*})$ for which~$a<_{P^*}b$ if and
only if~$b<_P a$.  We now arrive at our object of study:
\begin{defn}
  The \emph{circular Bruhat order}~$\cb(k,n)$ is the poset~$\bound(k,n)^*$.  
\end{defn}

The circular Bruhat order was originally defined in~\cite{Postnikov} in terms of
\emph{decorated permutations}.  These are permutations~$\pi\in S_n$ for which
each fixed point is assigned a color---either black or white.  The
\emph{anti-excedances} of a decorated permutation~$\pi$ are~$i\in[n]$ for which
either~$\pi^{-1}(i)>i$ or~$i$ is a white fixed point.  Then~$\cb(k,n)$ was
defined to be the set of decorated permutations with~$k$ anti-excedances and
with a poset structure determined by \emph{alignments}
and~$\emph{crossings}$ in \emph{chord diagrams}.  See~\cite{Postnikov} for
details. To go from a bounded affine permutation~$f$ to a decorated
permutation~$\pi$, reduce the window of~$f$ modulo~$n$, and then color each
fixed point~$i$ in the resulting permutation black if~$f(i)=i$ or white
if~$f(i)=i+n$.  We translate the notion of an anti-excedance from decorated
permutations to bounded affine permutations:

\begin{defn}
  The \emph{anti-excedances} of a bounded affine permutation $f\in \cb(k,n)$ are
  the integers $f(i)-n$  such that
  $i\in[n]$ and~$f(i)>n$.
\end{defn}
\noindent One may check that the elements of~$\cb(k,n)$ are exactly the bounded
affine permutations with~$k$ anti-excedances. 

To describe the poset structure on~$\cb(k,n)$ in detail, note that for
a reflection of an element of~$\cb(k,n)$ to remain in~$\cb(k,n)$, it is necessary
(but not sufficient) that the integer~$r$ in~\eqref{eqn:reflections} be~$0$
or~$1$.  Thus, for~$i,j\in[n]$ with~$i\neq j$, we define
\[
  t_{ij}=
  \begin{cases}
    [1,2,\ldots,i-1,j,i+1,\ldots,j-1,i,j+1,\ldots,n]&\text{if~$i<j$},\\
    [1,2,\ldots,j-1,i-n,j+1,\ldots,i-1,j+n,i+1,\ldots,n]&\text{if~$i>j$}.
  \end{cases}
\]
The cover relations in~$\cb(k,n)$ are given by $g\lessdot f$ if and only if there exists~$t_{ij}$ such that~$g
=ft_{ij}$ and~$\ell(f)=\ell(g)-1$ (recalling that~$\cb(k,n)$ is the dual
of~$\bound(k,n)$). 

By Lemma~17.6 of~\cite{Postnikov}, the unique maximal element of~$\cb(k,n)$
is
\[
 f_{\mathrm{top}} := [1+k,2+k,\ldots,n+k].
\] 
The minimal elements are in bijection with~$\binom{[n]}{k}$.
Given~$\lambda\in\binom{[n]}{k}$, the corresponding minimal element
is
\[
  f_{\mathrm{min},\lambda}(i)=
  \begin{cases}
    i+n&\text{if $i\in\lambda$,}\\
    i&\text{otherwise.}
  \end{cases}
\]
We have $\ell(f_{\mathrm{top}})=0$,
and~$\ell(f_{\mathrm{min},\lambda})=k(n-k)$ for any minimal element.  Thus, the rank function
for~$\cb(k,n)$ is $\rk(f) = k(n-k)-\ell(f)$.  By Proposition~23.1
of~\cite{Postnikov}, the exponential generating function for the cardinality
of~$\cb(k,n)$ is 
\[
  \sum_{0\leq k\leq
  n}|\cb(k,n)|x^{k}\frac{y^{n}}{n!}=e^{xy}\frac{x-1}{x-e^{y(x-1)}}.
\]
For the rank generating function of~$\cb(k,n)$, see~\cite{Williams}.

\section{Main theorem}\label{sect:main}
\begin{defn}\label{def:weights} Let~$\alpha_1,\ldots,\alpha_n$ be indeterminates.
  The \emph{weight}  of a covering~$ft_{ij}\lessdot f$ in~$\cb(k,n)$ is
  the sum of~$\alpha_i$ through~$\alpha_{j-1}$ in cyclic order:
  \[
    \wt(ft_{ij}\lessdot f)=
    \begin{cases}
      \alpha_i+\alpha_{i+1}+\cdots+\alpha_{j-1}&\text{if $i<j$},\\
      \alpha_{i}+\cdots+\alpha_n+\alpha_1+\alpha_2+\cdots+\alpha_{j-1}&\text{if~$i>j$}. 
    \end{cases}
  \]
  The \emph{weight} of a saturated chain in~$\cb(k,n)$ is the product of the
  weights of its cover relations (the empty chain is assigned weight~$1$).
  For~$r\in[n]$, a covering is~\emph{$r$-good} if~$\alpha_r$ appears in its
  weight.  A saturated chain in~$\cb(k,n)$ is \emph{$r$-good} if all of its
  cover relations are~$r$-good.  For arbitrary~$r\in \Z$, we define~$r$-good
  covers and chains by replacing~$r$ with its representative in~$[n]$ modulo~$n$.
\end{defn}

Our main theorem is the following.
\begin{thm}\label{thm:main}
  The sum of the weights of the maximal chains in~$\cb(k,n)$ is
  \[
    f(k,n)(\alpha_1+\cdots+\alpha_{n})^{k(n-k)}
  \]
  where~$f(k,n)$ is the number of standard Young tableaux of a~$k\times (n-k)$ rectangle.
\end{thm}

\begin{example}\label{example:cb(2,3)}
  Figure~\ref{fig:CB(2,3)} illustrates $\cb(2,3)$ with its cover weights.  The
  sum of the weights of its six maximal chains is
  \[
    \alpha_1\alpha_2+\alpha_1(\alpha_1+\alpha_3)+\alpha_2(\alpha_1+\alpha_2)+\alpha_2\alpha_3
    +\alpha_3\alpha_1+\alpha_3(\alpha_2+\alpha_3)=f(2,3)(\alpha_1+\alpha_2+\alpha_3)^{2},
  \]
  where~$f(2,3)=1$ since there is only one Young tableau for the~$2\times 1$
  rectangle.
\end{example}

\begin{figure}[h]
  \centering
  \begin{center}
    \def\x{5}
    \def\y{2}
    \begin{tikzpicture}
      \node at (0,0) (11) {$[3,4,5]$};

      \node at (-\x,-\y) (21) {$[4,3,5]$};
      \node at (0,-\y) (22) {$[3,5,4]$};
      \node at (\x,-\y) (23) {$[2,4,6]$};

      \node at (-\x,-2*\y) (31) {$[4,5,3]$};
      \node at (0,-2*\y) (32) {$[4,2,6]$};
      \node at (\x,-2*\y) (33) {$[1,5,6]$};

      \draw (11)--(21) node[midway, left=2.5mm] {$\alpha_1$};
      \draw (11)--(22) node[midway, left] {$\alpha_2$};
      \draw (11)--(23) node[midway, right=1.5mm] {$\ \alpha_3$};

      \draw (21)--(31) node[midway, left=0mm] {$\alpha_2$};
      \draw (21)--(32) node[pos=-0.03mm, right=5mm,rotate=-20] {$\alpha_1+\alpha_3$};
      \draw (22)--(31) node[pos=-0.01mm, left=5mm,rotate=20] {$\alpha_1+\alpha_2$};
      \draw (22)--(33) node[pos=0.1, right=3mm] {$\alpha_3$};
      \draw (23)--(32) node[pos=0.1, left=3mm] {$\alpha_1$};
      \draw (23)--(33) node[midway, right=0mm] {$\alpha_2+\alpha_3$};
    \end{tikzpicture}
  \end{center}
  \caption{$\cb(2,3)$ with edge weights.}
  \label{fig:CB(2,3)}
\end{figure}
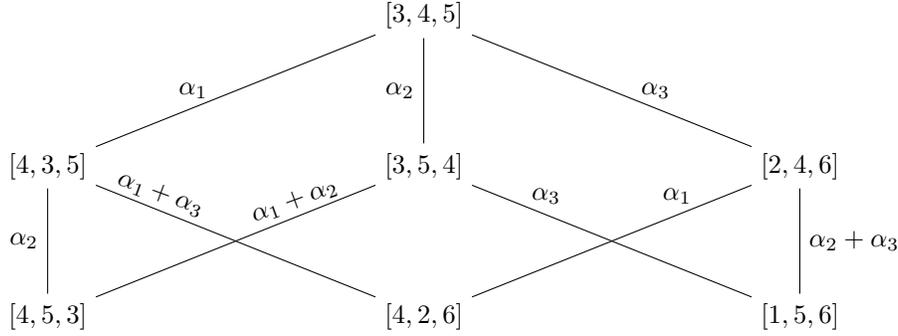
The proof of Theorem~\ref{thm:main} follows from two lemmas whose proofs appear
in next section.

\begin{lemma}\label{lemma:independence} Let~$f\in\cb(k,n)$. Then the number of
  $r$-good downward-saturated chains in~$\cb(k,n)$ with maximal element~$f$ is
  independent of~$r$.
\end{lemma}

\begin{lemma}\label{lemma:count}
  The number of $n$-good maximal chains in~$\cb(k,n)$ is~$f(k,n)$.
\end{lemma}

\begin{proof}[Proof of Theorem~\ref{thm:main}]
  Let~$\delta(f)$ be the number of~$r$-good downward-saturated chains with
  maximal element $f\in\cb(k,n)$.  This number is independent of~$r$ by
  Lemma~\ref{lemma:independence}.  Noting  linearity in the~$\alpha_i$, it
  is straightforward to check that
  \begin{equation}\label{eqn:induct}
    \sum_{t_{ij}:ft_{ij}\lessdot f}\delta(ft_{ij})\wt(ft_{ij}\lessdot f)
    = \delta(f)(\alpha_{1}+\cdots+\alpha_{n}).
  \end{equation}

  Let~$\mathcal{C}(m)$ denote the set of maximal chains~$C$ in~$\cb(k,n)$ such
  that~$\max(C)=f_{\mathrm{top}}$ and~$\ell(\min(C))=m$.  We now show by induction on~$m$
  that
  \[
    \sum_{C\in\mathcal{C}(m)}\delta(\min(C))\wt(C)=\delta(f_{\mathrm{top}})(\alpha_1+\cdots+\alpha_{n})^{m}
  \]
  for~$0\leq m\leq k(n-k)$.  The case~$m=0$ is a tautology since the only element
  of~$\cb(k,n)$ with length~$0$ is~$f_{\mathrm{top}}$.  On the other hand, in the
  case~$m=k(n-k)$ we are summing over maximal chains~$C$ of~$\cb(k,n)$.  For
  these~$\delta(\min(C))=1$, and Theorem~\ref{thm:main} will then follow from
  Lemma~\ref{lemma:count}.  To proceed with induction, fix some~$m$ with~$0\leq
  m< k(n-k)$.  Then
  \begin{align*}
    \sum_{C\in\mathcal{C}(m+1)}\delta(\min(C))\wt(C)
    &= \sum_{C'\in \mathcal{C}(m)}\sum_{f\lessdot \min(C')}\delta(f)\wt(f\lessdot
    \min(C'))\wt(C')
    \\[5pt]
    &= \sum_{C'\in
  \mathcal{C}(m)}\delta(\min(C'))(\alpha_1+\cdots+\alpha_{n})\wt(C')&\text{(by~\eqref{eqn:induct})}\\[6pt]
    &= \delta(f_{\mathrm{top}})(\alpha_1+\cdots+\alpha_{n})^{m+1}&\text{(by
    induction)}.
  \end{align*}
\end{proof}

\section{Proofs of lemmas}\label{sect:lemmas}

\subsection{\texorpdfstring{$k$}{\textbf{k}}-Bruhat order}
Our references for the~$k$-Bruhat order are~\cite{BS} and~\cite{KLS}.  
\begin{defn}\label{def:kbruhat} The \emph{$k$-Bruhat order}~$\leq_k$ on the symmetric group~$S_n$ is given
  by~$u\leq_k v$ if
\begin{enumerate}
  \item\label{item1-kbruhat} $u(i)\leq v(i)$ for $1\leq i\leq k$;
  \item\label{item2-kbruhat} $u(j)\geq v(j)$ for $k<j\leq n$;
  \item\label{item3-kbruhat} $u(i)<u(j)$ implies $v(i)<v(j)$ if $1\leq i<j\leq k$ or if $k<i<j\leq n$
\end{enumerate}
\end{defn}
The cover relations for the~$k$-Bruhat order have the form~$u\lessdot_k v$
if~$u\lessdot v$ (in the ordinary Bruhat order)
and~$\left\{u(1),\ldots,u(k)\right\}\neq\left\{v(1),\ldots,v(k)\right\}$.  Each
interval~$[u,w]_k$ in the $k$-Bruhat order is a graded poset of rank $\ell(w)-\ell(u)$.
\begin{defn}
  A permutation~$w\in S_n$ is \emph{$k$-Grassmannian} (or just
  \emph{Grassmannian} when~$k$ is clear from context) if $w(1)<\cdots<w(k)$
  and~$w(k+1)<\cdots<w(n)$.  
  These are in bijection with $\lambda\in\binom{[n]}{k}$ by letting $w_{\lambda}$ be
  the unique $k$-Grassmannian permutation such
  that~$\left\{w(1),\ldots,w(k)\right\}=\lambda$.  
\end{defn}

Denote the positions of the anti-excedances of $f\in\cb(k,n)$ by
\[
  \Lambda(f)=\left\{i\in[n]:f(i)-n \text{ is an anti-excedance of~$f$}\right\}
  = \left\{i\in[n]: f(i)>n\right\}.
\]
Then associate a $k$-Grassmannian permutation to~$f$ by
\[
  w_f=w_{\Lambda(f)}.
\]
Fixing~$\lambda\in\binom{[n]}{k}$, we define two posets.  The first is the
principal order ideal in the~$k$-Bruhat order generated by~$w_{\lambda}$:
\[
S_{n,\lambda}=\left\{u\in S_n:u\leq_k w_{\lambda}\right\}.
\]
The second is
\[
  \cb(k,n)_{\lambda}=\left\{f\in\cb(k,n):\Lambda(f)=\lambda\right\}
\]
with partial order~$\leq_{\gamma}$ defined by its cover
relations: $g\lessdot_{\gamma} f$ if~$g$ is covered
by~$f$ in~$\cb(k,n)$ and the covering~$g\lessdot f$ is~$n$-good ($\gamma$ is a
mnemonic for ``good'').  Note that a
covering~$ft_{ij}\lessdot f$ in $\cb(k,n)$ is $n$-good if and only if~$i>j$, in
which case~$\Lambda(ft_{ij})=\Lambda(f)$.

Embed~$S_n$ in~$\tS$ via~$u\mapsto[u(1),\ldots,u(n)]$ and
define the \emph{translation element}~$t_k=[1+n,2+n,\ldots,k+n,k+1,k+2,\ldots,n]\in\tSk$.
Taking our lead from~\cite{KLS}, for
each~$u\in S_{n,\lambda}$ define $f_{u}=f_{u,\lambda}=ut_kw_{\lambda}^{-1}$.
Therefore,
\[
  f_u(w_{\lambda}(i))=
  \begin{cases}
    u(i)+n&\text{if~$1\leq i\leq k$}\\
    u(i)&\text{if~$k<i\leq n$}.
  \end{cases}
\]
Since~$w_{\lambda}$ is $k$-Grassmannian and~$u\leq_k w_{\lambda}$, we have
$1\leq u(i)\leq w_{\lambda}(i)\leq n$ for~$1\leq i\leq k$, and~$w_{\lambda}(i)\leq u(i)\leq n$ for~$k<i\leq n$.
Therefore, $i\leq f(i)\leq i+n$ for all~$i$.  Further, $\Lambda(f_u)=\lambda$.
Hence,~$f_u\in\cb(k,n)_{\lambda}$.

For each~$f\in\cb(k,n)_{\lambda}$,
define~$u_{f}=u_{f,\lambda}=fw_{\lambda}t_k^{-1}$ so that
\[
  u_f(i)=
  \begin{cases}
    f(w_{\lambda}(i))-n&\text{if $1\leq i\leq k$}\\
    f(w_{\lambda}(i))&\text{if~$k<i\leq n$}.
  \end{cases}
\]
Since~$\lambda=\Lambda(f)$, it follows that~$1\leq u_f(i)\leq n$ for~$i\in[n]$.
To see that $u\leq_k w_{\lambda}$, first note that  $w_{\lambda}(i)\leq
f(w_{\lambda}(i))\leq w_{\lambda}(i)+n$ for all~$i$ since~$f\in\cb(k,n)$.
Properties (\ref{item1-kbruhat}) and (\ref{item2-kbruhat}) of
Definition~\ref{def:kbruhat} then follow.  Property~\ref{item2-kbruhat} holds
since~$w_{\lambda}$ is a $k$-Grassmannian element and, therefore,
$w_{\lambda}(i)$ is increasing for $1\leq i\leq k$ and for~$k<i\leq n$.

\begin{example}  Let~$\lambda=\left\{2,4,5\right\}\in\binom{[5]}{3}$ and
  $f=[3,6,5,9,7]$.
  Then~$w_{\lambda}=[2,4,5,1,3]$, and~$f\in\cb(3,5)_{\lambda}$ since its anti-excedances appear in
  positions~$2,4,$ and~$5$.  We have~$u_f=[1,4,2,3,5]$, which is formed by first
  listing the anti-excedances of~$f$, reduced modulo~$5$, as they appear in
  order by position in~$f$, i.e., $1=6-5$, $4=9-5,$ and~$2=7-5$, and then listing the
  non-anti-excedances,~$3$ and~$5$.  Reversing this process yields~$f_{u_f}=f$.
\end{example}

\begin{prop}\label{prop:anti-isomorphism}\footnote{For a closely related result,
  see \cite[Theorem 3.16]{KLS}.}
  Let~$\lambda\in\binom{[n]}{k}$.  Then the mapping
  \begin{align*}
    (S_{n,\lambda},\leq_k)&\to(\cb(k,n)_{\lambda},\leq_{\gamma})\\
    u&\mapsto f_u
  \end{align*}
  is an anti-isomorphism of posets with inverse~$f\mapsto u_f$. 
\end{prop}
\begin{proof} It is clear that $u\mapsto f_u$ and $f\mapsto u_{f}$ are inverses.
  We must show they reverse cover relations.
  Let~$u,v\in S_{n,\lambda}$ with corresponding~$f:=f_u$ and~$g:=f_v$ in
  $\cb(k,n)$. 
  The condition that~$u\lessdot_k v$ is equivalent to:
  \begin{enumerate}[label=\arabic*., ref=\arabic*]
    \item\label{prop:item1} There exists~$p\leq k<q$ such that~$v=us_{p,q}$ where~$s_{p,q}=(p,q)$ is the
      transposition swapping~$p$ and~$q$, and
    \item\label{prop:item2} $\ell(v)=\ell(u)+1$, i.e.,
      \begin{enumerate}[label=(\roman*), ref=(\roman*)]
	\item\label{prop:item2i} $u(p)<u(q)$, and
	\item\label{prop:item2ii} there is no integer~$r$ such that~$p<r<q$ and $u(p)<u(r)<u(q)$.
      \end{enumerate}
  \end{enumerate}
  On the other hand, the condition that~$g\lessdot_{\gamma}f$ is equivalent to:
  \begin{enumerate}[label=\arabic*$^{\ast}$., ref = \arabic*$^{\ast}$]
    \item\label{prop:inv1} There exists $i<j$ such that~$f(i)$ is a non-anti-excedance, $f(j)$ is
      an anti-excedance,~$g=ft_{ji}$, and
    \item\label{prop:inv2} $\ell(g)=\ell(f)+1$, i.e.,
      \begin{enumerate}[label=(\roman*), ref=(\roman*)]
	\item\label{prop:inv2i} $f(j)<f(i)+n$, and
	\item\label{prop:inv2ii} there is no integer~$a$ such that~$j<a<i+n$
	  and~$f(j)<f(a)<f(i)+n$.
      \end{enumerate}
  \end{enumerate}

  To show equivalence of these two sets of conditions, first suppose
  that~$u\lessdot_k v\leq_k w_{\lambda}$. We will show
  that~$g\lessdot_{\gamma}f$. Take~$p\leq k<q$ as in condition~\ref{prop:item1},
  and let~$i:=w_{\lambda}(q)$ and~$j:=w_{\lambda}(p)$.  It follows from
  the~$k$-Bruhat order that~$i<j$:
  \[
    i=w_{\lambda}(q)\leq v(q)=u(p)\leq v(p)\leq w_{\lambda}(p)=j.
  \]
  We have that~$f(r)=g(r)$ for~$r\in[n]\setminus\left\{i,j\right\}$, and
  \begin{align*}
    g(i)&=g(w_{\lambda}(q))=v(q)=u(p)=f(w_{\lambda}(p))-n=f(j)-n\\
    g(j)&=g(w_{\lambda}(p))=v(p)+n=u(q)+n=f(w_{\lambda}(q))+n=f(i)+n.
  \end{align*}
  Therefore, condition~\ref{prop:inv1} holds, and~\ref{prop:inv2}\ref{prop:inv2i} follows from~\ref{prop:item2}\ref{prop:item2i}.

  Condition~\ref{prop:inv2}\ref{prop:inv2ii} says that the graph of~$f$ has
  no points inside a certain box:
  
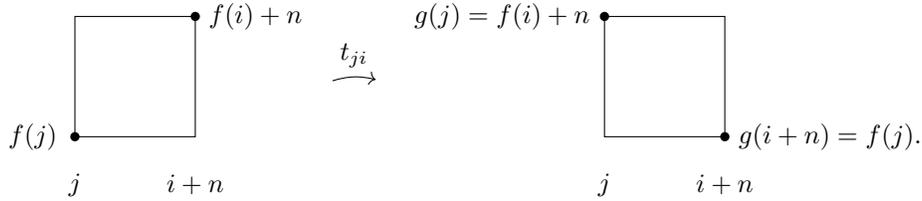
\begin{center}
    \begin{tikzpicture}[scale=0.8]
      \draw (0,0) rectangle (2,2);
      \node at (0,-0.8) {$j$};
      \node at (2,-0.8) {$i+n$};
      \node at (-0.7,0) {$f(j)$};
      \node at (3,2) {$f(i)+n$};
      \node[dot] at (0,0) {};
      \node[dot] at (2,2) {};
      \draw (4.28,0.93) edge[->,bend left=15,above] node {$t_{ji}$} (5,0.95);
      \begin{scope}[xshift=8.8cm]
	\draw (0,0) rectangle (2,2);
	\node at (0,-0.8) {$j$};
	\node at (2,-0.8) {$i+n$};
	\node at (-1.7,2) {$g(j)=f(i)+n$};
	\node at (3.75,0) {$g(i+n)=f(j)$.};
	\node[dot] at (0,2) {};
	\node[dot] at (2,0) {};
      \end{scope}
    \end{tikzpicture}
  \end{center}
  To verify condition~\ref{prop:inv2}\ref{prop:inv2ii} holds, it helps to
  divide the sequence of integers $j,j+1,\ldots,i+n$ into two parts:
  $X:=\left\{a \in\Z: j<a\leq n\right\}$, and $Y:=\left\{a\in\Z:n<a<
  i+n\right\}$.  If~$a\in X$ and~$f(a)$ is not an anti-excedance, then
  $f(a)<n<f(j)$ and, hence, condition~\ref{prop:inv2}\ref{prop:inv2ii} is not
  violated.  Similarly, if~$a\in Y$ and~$f(a-n)$ \emph{is} an
  anti-excedance, then~$f(i)+n\leq 2n<f(a-n)+n=f(a)$ and,
  hence,~\ref{prop:inv2}\ref{prop:inv2ii} is again not violated.

  It remains to check anti-excedances whose positions are in~$X$ and
  non-anti-excedances whose positions are between~$1$ and~$i$ (i.e., are
  in~$-n+Y$). Take~$a\in X$ and suppose that~$f(a)$ is an anti-excedance.
  Since~$a>j=w_{\lambda}(p)$, there exists~$r$ with $p<r\leq k<q$ such that
  $f(a)=u(r)+n$.  By condition~\ref{prop:item2}\ref{prop:item2ii},~$u(r)$ is
  not between~$u(p)$ and~$u(q)$, which implies that~$f(a)$ is not
  between~$f(i)+n=u(q)+n$ and~$f(j)=u(p)+n$ in accordance with~\ref{prop:inv2}\ref{prop:inv2ii}.  Now,
  instead, take~$a\in Y$ and suppose that~$f(a-n)$ is not an anti-excedance.
  Since~$a-n<i=w_{\lambda}(q)$, there exists~$r$ with~$p\leq k<r<q$ such that $f(a-n)=u(r)$.  As
  above,~$u(r)$ is not between~$u(p)$ and~$u(q)$, and this implies that
  $f(a)=u(r)+n$ is not between~$f(i)+n=u(q)+n$ and~$f(j)=u(p)+n$.

  We have shown that the mapping $u\mapsto f_u$ reverses cover relations.
  The proof that its inverse~$f\mapsto u_f$ reverses cover relations is similar.
\end{proof}

\begin{cor}\label{cor:chains} Let~$f\in \cb(k,n)$ with corresponding Grassmannian
  permutation~$w_f$.  Then the~$n$-good downward-saturated chains in~$\cb(k,n)$
  with maximal element~$f$ are in bijection with the maximal chains in
  the~$k$-Bruhat interval~$[u_f,w_f]_k$.  
\end{cor}
\begin{proof} Let~$\lambda=\Lambda(f)$. Since an $n$-good covering preserves
  anti-excedance positions, downward-saturated $n$-good chains in $\cb(k,n)$
  with maximal element~$f$ are exactly downward saturated chains
  in~$\cb(k,n)_{\lambda}$ with maximal element~$f$.  By
  Proposition~\ref{prop:anti-isomorphism} these are in bijection with upward
  saturated chains in $S_{n,\lambda}$ with minimal element~$u_f$.  Since
  $S_{n,\lambda}$ has unique maximal element~$w_{\lambda}=w_{f}$, these
  upward-saturated chains are exactly the maximal chains in~$[u_f,w_f]_{k}$.
\end{proof}

\begin{example}
  Let~$f=[2,5,4,7]\in\cb(2,4)$.  Then~$w_f=[2,4,1,3]\in S_n$.  We
  have $u_f=[1,3,2,4]$ and $f_{w_{f}}=[1,6,3,8]$.  As seen in
  Figure~\ref{fig:intervals}, the interval~$[u_f,w_f]_2$ has
  two maximal chains:
  \[
    [1,3,2,4]\lessdot_2[1,4,2,3]\lessdot_2[2,4,1,3]\quad\text{and}\quad[1,3,2,4]\lessdot_2[2,3,1,4]\lessdot_2[2,4,1,3].
  \]
  Under the isomorphism of Proposition~\ref{prop:anti-isomorphism}, these
  correspond to the two~$n$-good downward-saturated chains in~$\cb(2,4)$ with
  maximal element~$f$:
  \[
    [1,6,3,8]\lessdot[2,5,3,8]\lessdot[2,5,4,7]\quad\text{and}\quad[1,6,3,8]\lessdot[1,6,4,7]\lessdot[2,5,4,7].
  \]
\end{example}

\begin{figure}[ht]
  \centering
  \begin{center}
    \begin{tikzpicture}
      \def\x{2}
      \def\y{1.3}
      \node at (0,0) (11) {$[2,4,1,3]$};

      \node at (-\x,-\y) (21) {$[1,4,2,3]$};
      \node at (0*\x,-\y) (22) {$[2,1,4,3]$};
      \node at (\x,-\y) (23) {$[2,3,1,4]$};

      \node at (-\x,-2*\y) (31) {$[1,2,4,3]$};
      \node at (0*\x,-2*\y) (32) {$[1,3,2,4]$};

      \node at (0,-3*\y) (41) {$[1,2,3,4]$};

      \draw (11)--(21);
      \draw (11)--(22);
      \draw (11)--(23);

      \draw (21)--(31);
      \draw (21)--(32);
      \draw (23)--(32);

      \draw (32)--(41);

      \begin{scope}[xshift=7cm]
	\node at (0,0) (11) {$[3,5,4,6]$}; 

	\node at (-\x,-\y) (21) {$[4,5,3,6]$};
	\node at (0,-\y) (22) {$[2,5,4,7]$};

	\node at (-\x,-2*\y) (31) {$[2,5,3,8]$};
	\node at (0,0-2*\y) (32) {$[4,6,3,5]$};
	\node at (\x,0-2*\y) (33) {$[1,6,4,7]$};

	\node at (0,-3*\y) (41) {$[1,6,3,8]$};

       \draw (11)--(22);

       \draw (21)--(31);
       \draw (22)--(31);
       \draw (22)--(33);

       \draw (31)--(41);
       \draw (32)--(41);
       \draw (33)--(41);
      \end{scope}
    \end{tikzpicture}
  \end{center}
  
  \caption{The posets~$S_{4,\lambda}$ and $\cb(2,4)_{\lambda}$ in the
  case~$\lambda=\left\{2,4\right\}\in\binom{[4]}{2}$. They are anti-isomorphic
in accordance with Proposition~\ref{prop:anti-isomorphism}.}
  \label{fig:intervals}
\end{figure}
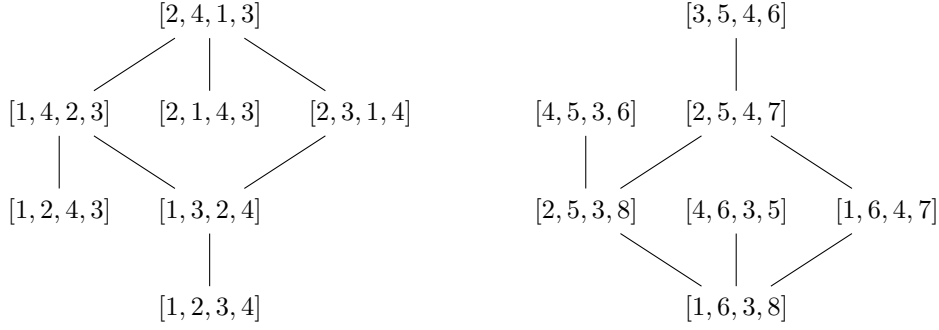

\subsection{Cyclic shifts} 
Define the \emph{cyclic shift} of~$f\in \cb(k,n)$ by
\[
  \chi(f) := [f(0)+1,f(1)+1,\ldots,f(n-1)+1].
\]
The following properties of~$\chi$ are immediate: (i)~$i\leq
\chi(f)(i)\leq i+n$ for all~$i\in[n]$, (ii)~$\chi(ft_{ij})=\chi(f)t_{i+1,j+1}$
for all reflections~$t_{i,j}$ (with indices taken modulo~$n$),  and
(iii)~$(i,j)\in\Z^{2}$ represents an inversion for~$f$ if and only
if~$(i+1,j+1)$ represents an inversion for~$\chi(f)$.  Therefore,~$\chi$ is an
automorphism of the graded poset~$\cb(k,n)$ and induces a faithful action of the
cyclic group of order~$n$ on~$\cb(k,n)$.  The following is an immediate
implication of property (ii).

\begin{prop}\label{prop:cyclic shift}
  A saturated chain~$C$ in~$\cb(k,n)$ is~$r$-good if and only if~$\chi(C)$
  is~$(r+1)$-good.
\end{prop}

The following result of Bergeron and Sottile is an important step in the proof
of Lemma~\ref{lemma:independence}.

\begin{thm}[{\cite[Corollary 1.3.1]{BS}}]\label{thm:BS}
  Let~$u\leq_k v$ and~$x\leq_k y$ in the~$k$-Bruhat order on~$S_n$,
  and suppose that~$cvu^{-1}c^{-1}=yx^{-1}$
  where~$c=[2,3,\ldots,n,1]=(1,2,\ldots,n)$. Then
  the intervals~$[u,v]_k$ and~$[x,y]_k$ have the same number of maximal chains.
\end{thm}

\subsection{Proofs of lemmas}
\begin{proof}[Proof of Lemma~\ref{lemma:independence}]  For general~$g\in\cb(k,n)$,
  let $\delta_r(g)$ denote the number of $r$-good downward-saturated chains
  in~$\cb(k,n)$ with
  maximal element~$g$.  Fix $f\in\cb(k,n)$.  We
  first show that Theorem~\ref{thm:BS} applies to the pair of
intervals~$[u_f,w_f]_k$ and~$[u_{\chi(f)},w_{\chi(f)}]_k$ by checking
  that~$cu_fw_f^{-1}=u_{\chi(f)}w_{\chi(f)}^{-1}c$ where~$c=[2,3,\ldots,n,1]$.
  Let~$i\in[n]$.  Working modulo~$n$, 
  \[
    cu_fw_f^{-1}(i)=cu_ft_kw_f^{-1}(i)=f(i)+1=\chi(f)(i+1)=u_{\chi(f)}t_kw_{\chi(f)}^{-1}c(i)=u_{\chi(f)}w_{\chi(f)}^{-1}c(i),
  \]
  and the conclusion follows.  Therefore, the number of maximal chains
  in~$[u_f,w_f]_k$ is the same as the number of maximal chains in
  $[u_{\chi(f)},w_{\chi(f)}]_k$ and, by induction, as the number of maximal
  chains in $[u_{\chi^r(f)},w_{\chi^r(f)}]_k$ for all~$r\in[n]$.  
  Applying Corollary~\ref{cor:chains} and Proposition~\ref{prop:cyclic shift},
  \[
    \delta_n(f)=\delta_n(\chi^r(f))=\delta_r(f)
  \]
  for all~$r\in[n]$.
\end{proof}

\begin{proof}[Proof of Lemma~\ref{lemma:count}]
  Let~$w_{\mathrm{max}}=w_{\{n-k+1,n-k+2,\ldots,n\}}$ and
  $\mathrm{id}=[1,2,\ldots,n]$. Then~$w\leq_{k}w_{\mathrm{max}}$ for
  all~$k$-Grassmannian permutations~$w$ and~$\mathrm{id}\leq_{k}u$ for all~$u\in
  S_n$.  Thus, by Corollary~\ref{cor:chains}, maximal~$n$-good chains in
  $\cb(k,n)$ are in correspondence with maximal chains
  in~$[\mathrm{id},w_{\mathrm{max}}]_{k}$, an interval of rank~$k(n-k)$. This
  interval is exactly the set of all~$k$-Grassmannian elements of~$S_n$.

  Let~$L(k,n-k)$ denote the poset of Young diagrams fitting inside a~$k\times
  (n-k)$ rectangle, ordered by containment, as usual.  There is a well-known
  correspondence between~$\binom{[n]}{k}$ and~$L(k,n-k)$:
  given~$\lambda\in\binom{[n]}{k}$
  with~$\lambda=\left\{\lambda_1<\cdots<\lambda_k\right\}$, let~$Y_{p(\lambda)}$
  be the Young diagram corresponding to the partition~$p(\lambda)=\{
  p_1>\cdots>p_k\}$ where~$p_i:=(n-k)-\lambda_i+i$. In English
  notation,~$Y_{p(\lambda)}$ is the diagram determined by the left-down walk
  in~$\Z^2$ from~$(k,n-k)$ to~$(0,0)$ whose $\lambda_i$-th step is its $i$-th
  \emph{vertical} step.  This correspondence yields an
  anti-isomorphism~$w_{\lambda}\mapsto Y_{p(\lambda)}$ from the
  interval~$[\mathrm{id},w_{\mathrm{max}}]_k$ with its $k$-Bruhat order and
  $L(k,n-k)$.  The result now follows since Young tableaux for
  the~$k\times(n-k)$ rectangle are in bijection with maximal chains in
  $L(k,n-k)$. 
\end{proof}

\bibliographystyle{alpha}
\bibliography{cb}

\newcommand{\etalchar}[1]{$^{#1}$}
\begin{thebibliography}{AHBC{\etalchar{+}}16}

\bibitem[AHBC{\etalchar{+}}16]{Arkani}
Nima Arkani-Hamed, Jacob Bourjaily, Freddy Cachazo, Alexander Goncharov,
  Alexander Postnikov, and Jaroslav Trnka.
\newblock {\em Grassmannian geometry of scattering amplitudes}.
\newblock Cambridge University Press, Cambridge, 2016.

\bibitem[BB96]{BB}
Anders Bj\"{o}rner and Francesco Brenti.
\newblock Affine permutations of type {$A$}.
\newblock {\em Electron. J. Combin.}, 3(2):Research Paper 18, approx. 35, 1996.
\newblock The Foata Festschrift.

\bibitem[BS98]{BS}
Nantel Bergeron and Frank Sottile.
\newblock Schubert polynomials, the {B}ruhat order, and the geometry of flag
  manifolds.
\newblock {\em Duke Math. J.}, 95(2):373--423, 1998.

\bibitem[GKL20]{GKL}
Pavel Galashin, Steven~N. Karp, and Thomas Lam.
\newblock Regularity theorem for totally nonnegative flag varieties.
\newblock {\em S\'{e}m. Lothar. Combin.}, 84B:Art. 31, 12, 2020.

\bibitem[GL]{GL}
Pavel Galashin and Thomas Lam.
\newblock Positroid varieties and cluster algebras.
\newblock To appear in \emph{Ann. Sci. \'Ec. Norm. Sup\'er.}

\bibitem[KLS13]{KLS}
Allen Knutson, Thomas Lam, and David~E. Speyer.
\newblock Positroid varieties: juggling and geometry.
\newblock {\em Compos. Math.}, 149(10):1710--1752, 2013.

\bibitem[KW11]{KW}
Yuji Kodama and Lauren~K. Williams.
\newblock K{P} solitons, total positivity, and cluster algebras.
\newblock {\em Proc. Natl. Acad. Sci. USA}, 108(22):8984--8989, 2011.

\bibitem[Lam18]{Lam}
Thomas Lam.
\newblock Electroid varieties and a compactification of the space of electrical
  networks.
\newblock {\em Adv. Math.}, 338:549--600, 2018.

\bibitem[Pos05]{Postnikov}
Alexander Postnikov.
\newblock Total positivity, {G}rassmannians, and networks.
\newblock Preprint (2005)
  \url{http://www-math.mit.edu/~apost/papers/tpgrass.pdf}, 2005.

\bibitem[Ste02]{Stembridge}
John~R. Stembridge.
\newblock A weighted enumeration of maximal chains in the {B}ruhat order.
\newblock {\em J. Algebraic Combin.}, 15(3):291--301, 2002.

\bibitem[Wil05]{Williams}
Lauren~K. Williams.
\newblock Enumeration of totally positive {G}rassmann cells.
\newblock {\em Adv. Math.}, 190(2):319--342, 2005.

\bibitem[Wil07]{Williams2}
Lauren~K. Williams.
\newblock Shelling totally nonnegative flag varieties.
\newblock {\em J. Reine Angew. Math.}, 609:1--21, 2007.

\end{thebibliography}

\end{document}